\documentclass[A4paper, 12pt]{amsart}
\usepackage[latin1]{inputenc} 
\usepackage{amssymb,bbm,amsmath,amsthm, wasysym}
\usepackage{latexsym}

\usepackage[all]{xy}

\usepackage{graphicx}

\usepackage{epic} 
\usepackage{eepic}
\usepackage{xcolor}
\usepackage{hyperref}
\usepackage{verbatim}

\input xy
\xyoption{all}

\DeclareMathOperator{\Hom}{Hom}
\DeclareMathOperator{\Ext}{Ext}
\DeclareMathOperator{\Tor}{Tor}

\newcommand{\spec}[0]{\operatorname{Spec}}

\newtheorem{Thm}{Theorem}[section]

\newtheorem{Cor}[Thm]{Corollary}

\newtheorem{Lem}[Thm]{Lemma}

\newtheorem{Prop}[Thm]{Proposition}

\theoremstyle{definition}

       \newtheorem{Rmk}[Thm]{Remark}

        \newtheorem{lemma}[Thm]{Lemma}

\title{Local Cohomology and Base Change}
\author{Karen E. Smith}
\date{\today}

\begin{document}

\begin{abstract}
 Let  $X \overset{f}\longrightarrow S$ be  a  morphism of  Noetherian schemes, with  $S$ reduced. For any closed   subscheme  $Z$  of $X$ finite over $S$, let  $j$ denote the open immersion $X\setminus Z \hookrightarrow X$.  Then for any coherent sheaf 
  $\mathcal F$  on $X\setminus Z$ and any index $r\geq 1$,   the sheaf  $f_*(R^rj_*\mathcal F)$ is generically free on $S$ and commutes with base change. We prove this by proving a related statement about local cohomology:   Let $ R$ be  Noetherian  algebra over a Noetherian domain $A$,  and let $I \subset R$ be an ideal  such that $ R/I $ is  finitely generated as an $A$-module. Let $M$ be a   finitely generated $R$-module. Then there exists a non-zero $g \in A$ such that  the local cohomology modules $H^r_I(M) \otimes_A A_g$ are   free over $A_g$ and for any ring map  $A\rightarrow L$ factoring through $A_g$, we have
$H^r_I(M) \otimes_A L \cong H^r_{I{\otimes_A}L}(M\otimes_A L)$ for all $r$.  \end{abstract}




\thanks{Thanks to J\'anos Koll\'ar for  encouraging me to write this paper,  for his insistence that I establish the more general version of the theorem,  and for sharing his insight as to why my arguments should go through more generally, especially the idea to reduce to the affine case in Section 2. I'd also like to thank 
  Mel Hochster for listening to my arguments,  Karl Schwede for reading a draft and suggesting the reference \cite{DGI}, and Barghav Bhatt for his interest.
  Financial support partially from NSF 
   DMS-1501625.}

\maketitle

\section{Introduction}
In his work on maps between local Picard groups, Koll\'ar was led to investigate the behavior of certain cohomological functors  under base change  \cite{kollar}. The following theorem directly answers a question he had posed: 

\begin{Thm}
 \label{free+comm.main.ptop.karen}   Let $X\overset{f}\to S$ be  a  morphism of Noetherian schemes, with  $S$ reduced. Suppose that $Z\subset X$ is closed subscheme finite over $S$, and let $j$ denote the  open embedding of its complement $U$. 
 Then for any coherent sheaf   $\mathcal F$  on $U,$ the
sheaves 
 $f_*(R^rj_*\mathcal F)$ are generically free and commute with base change for all $r\geq 1$.
\end{Thm}

  Our purpose in this note is to  prove this general statement. 
  Koll\'ar  himself had proved a special case  of this result in a more restricted setting  \cite[Thm 78]{kollar}.

We pause to say precisely what is meant by  {\it generically free and commutes with base change}.
Suppose $\mathcal H$ is a functor which,  for every morphism of schemes $X\rightarrow S$ and every quasi-coherent sheaf $\mathcal F$ on $X$, produces a quasi-coherent sheaf $\mathcal H(\mathcal F)$ on $S$.  We say $\mathcal H(\mathcal F)$ is {\it generically free} if  there exists a dense open set $S^0$ of $S$ over which the  $\mathcal O_S$-module
$\mathcal H(\mathcal F) $ is free.   If in addition, 
for every change of base  $T\overset{p}\longrightarrow S^{0}$, the natural map 
$$ p^*\mathcal H(\mathcal F) \rightarrow \mathcal H(p^*_X\mathcal F)
$$  of quasi-coherent  sheaves on $T$  is an isomorphism
(where  $p_X$ is the induced morphism $X\times_S T \rightarrow X$),  then we say that $\mathcal H(\mathcal F) $ is {\it generically free and commutes with base change}.
See \cite[\S 72]{kollar}.

\begin{Rmk}  We do not claim the $r=0$ case of Theorem \ref {free+comm.main.ptop.karen}; in fact, it is false. For a counterexample, consider the ring
 homomorphism splitting $\mathbb Z \hookrightarrow  \mathbb Z \times \mathbb Q \twoheadrightarrow   \mathbb Z.$ The
corresponding morphism of Noetherian schemes 
 $$Z = \spec (\mathbb Z) \hookrightarrow  X = \spec (\mathbb Z \times \mathbb Q)  \rightarrow S = \spec   \mathbb Z$$ satisfies the hypothesis of Theorem 
 \ref{free+comm.main.ptop.karen}. The open set $U = X \setminus Z$ is the component  $\spec \mathbb Q$ of $X$. The coherent sheaf determined by the module $\mathbb Q$ on $U$ is not generically free over $\mathbb Z$, since there is no open affine subset $\spec \mathbb Z[\frac{1}{n}]$ over which $\mathbb Q$ is a free module.  [In this case, the map $j$ is affine, so the higher direct image sheaves $R^pj_* \mathcal F$ all vanish for $p >0$.]

 On the other hand,  if $f$ is a map of finite type, then the $r=0$ case of Theorem  \ref {free+comm.main.ptop.karen}
  can be deduced from Grothendiecks's Lemma on Generic freeness; see Lemma \ref{genfree}.  
\end{Rmk}
  
  For the commutative algebraists, we record the following version of the main result, which is essentially just the statement in the affine case:
  \begin{Cor}\label{commalgcor}
  Let $A$ be a reduced Noetherian ring. Let $R$ be a Noetherian $A$-algebra with ideal $I \subset R$ such that the  induced map $A \rightarrow R/I$ is finite.
  Then for any Noetherian $R$ module $M$,   the local cohomology
  modules $
H^i_I(M)
$
are generically free and commute with base change over $A$ for all $i\geq 0$.  Explicitly, this means that there exists element $g$ not in any minimal prime of $A$ such that  the modules 
  $H^i_I(M) \otimes_A A_g$ are free over $A_g$, and that  for any algebra $L$ over $A_g$, the natural map 
$$
H^i_I(M) \otimes_A L \rightarrow H^i_I(M \otimes_A L) 
$$
  is an isomorphism. 
  \end{Cor}
  
Some version of Theorem \ref{free+comm.main.ptop.karen}  may follow from known statements in the literature, but looking through works of Grothendieck (\cite{ResiduesAndDuality},  \cite{LocalCohomology}, \cite{SGA}) and  \cite{Conrad}, 
 I am not able to find it; nor presumably could Koll\'ar.   After this paper was written, I did find  a  related statement due to Hochster and Roberts \cite[Thm 3.4]{HR} in a special case, not quite strong enough to directly answer Koll\'ar's original question; furthermore, my proof is different and possibly of independent interest. In any case,  there may be value in the self-contained proof here, which 
 uses  a relative form of Matlis duality proved here using  
 only  basic results about  local cohomology  well-known to most commutative algebraists.
 
\section{Restatement of the Problem}

In this section, we show that 
Theorem \ref{free+comm.main.ptop.karen}  reduces to  the following special  statement:

\begin{Thm}\label{main2}
Let $A$ be a Noetherian domain.  Let $R = A[[x_1, \dots, x_n]]$ be a power series ring over $A$, and    $M$ a finitely generated $R$-module. 
Denote by $I$ the ideal $(x_1, \dots, x_n) \subset R$.  Then the local cohomology modules 
$$
H^i_I(M)
$$
are generically free over $A$ and commute with base change for all $i$.  
\end{Thm}

For basic definitions and properties of local cohomology modules, we refer to  \cite{LocalCohomology}.

\medskip

For the remainder of this section, we show how Theorem  \ref{free+comm.main.ptop.karen} and Corollary \ref{commalgcor} follow from Theorem  \ref{main2}.

First,  
Theorem \ref{free+comm.main.ptop.karen} is local on the base. 
 Because the scheme $S$ is reduced, it is the finite union of its irreducible components, each of which is reduced and irreducible, so it suffices to prove the result on each of them. 
 Thus we can immediately reduce to the case where $S = \spec A,$ for some  Noetherian domain $A$.

 We now reduce to the case where $X$ is affine as well. The coherent sheaf $\mathcal F$ on the open set $U$ extends to a coherent sheaf on $X$, which we also denote by $\mathcal F$. 
   To simplify notation, let us denote the sheaf 
 $R^rj_*\mathcal F$  by $\mathcal H$, which we observe vanishes outside the closed set $Z$.  Each section is annihilated by a power of the ideal $\mathcal I_Z$ of $Z$, so that  although the sheaf of abelian groups  $\mathcal H$ on $Z$ is not an $\mathcal O_Z$-module,  it does have the structure of a module over the sheaf of $\mathcal O_S$-algebras   $\varprojlim_t \frac{\mathcal O_X}{\mathcal I_Z^t}$, 
  which we denote $\widehat { \mathcal O_{X, Z};} $\,\, put differently, $\mathcal H$ can be viewed as sheaf on the formal scheme $\widehat X^{Z}$ over $S$.
Observe that  $i_*\widehat { \mathcal O_{X, Z}}_{|X^0} = 
 \widehat { \mathcal O_{X, Z}}$, \,  where $ X^0 \overset{i}\hookrightarrow X$ is the inclusion of any open set  $X^0 \subset X$ containing the generic points of all the components of $Z$. 
  This means that $\mathcal H$  is generically free as an  $\mathcal O_S$-module if and only if $\mathcal H_{|X^0}$  is, and there is 
  no loss of generality in replacing  $X$ by any such open set $X_0$.   In particular, we can choose such  $X^0$ to be affine, thus reducing the proof of Theorem \ref{free+comm.main.ptop.karen}  to the case where both $X$ and $S$ are affine.  

  We can now assume that $X \rightarrow S$ is the affine map of affine schemes corresponding to a ring homomorphism
  $A \rightarrow T$. In this case the  closed set $Z $ is defined by some ideal $I$ of $ T$. Because $Z$ is finite over $S = \spec A$, the
  composition $A \rightarrow T \rightarrow  T/I$  is a finite integral extension of $A$.   The  coherent sheaf $\mathcal F$ on $U$ extends  to a coherent sheaf $\mathcal F$ on $X$, which corresponds to a finitely generated $T$-module $M$.
Since $X = \spec T$ is affine, we have natural identifications  for $r \geq 1$
$$R^rj_*\mathcal F = H^{r}(X\setminus Z, \mathcal F) = H^{r+1}_I(M)$$ 
of modules over $T$  \cite[Cor 1.9]{LocalCohomology}. Thus we have reduced Theorem  \ref{free+comm.main.ptop.karen} to showing that if $T$ is a Noetherian ring over a Noetherian domain $A$ and $I$ is any ideal such that $T/I$ is  finitely generated as an $A$-module, then for any finitely generated $T$-module $M$, the modules $H^{r+1}_I(M)$ are generically free and commute with base change over $A$ for $ r \geq 1$. In fact, we will be able to show this for all indices $r\geq -1$.

To get to the power series case, we first observe that for all $i$, 
 every element of  $H^{i}_I(M)$  is annihilated by some power of $I$. This means that   $H^{i}_I(M)$ has the structure of a module over the $I$-adic completion $\hat T^I$.  
 There is no loss of generality in replacing $T$ and $M$  by their $I$-adic completions $\hat T^I$ and $\hat M^{ I}$---the module $H^{i}_I(M)$ is  canonically identified with $H^{i}_{I\hat T^ I}(\hat M^{I})$.  
 So with out loss of generality, we assume that $T$ is $I$-adically complete. 
 
 Now, Lemma \ref{power} below guarantees that $T$ is a module-finite algebra over a power series ring $
A[[x_1, \dots, x_n]] $. So the finitely generated $T$-module $M$ is also a finitely generated module over $A[[x_1, \dots, x_n]], $
and the computation of local cohomology is the same viewed over either ring. This means that to prove Theorem  \ref{free+comm.main.ptop.karen}, it  would suffice to prove Theorem \ref{main2}. It only remains to prove Lemma \ref{power}.

\begin{lemma}\label{power} Let $A$  be a Noetherian ring. Let $T$ be a Noetherian $A$-algebra containing an ideal $I$  such that the composition of natural maps
$ A \rightarrow T \twoheadrightarrow T/I $ is finite.  Then there is a natural  ring homomorphism from a power series ring
$$
A[[x_1, \dots, x_n]] \rightarrow \hat{ T}^I : = \varprojlim_t \, \, T/I^t$$
which is {\it also finite}.  
\end{lemma}

\begin{proof}[Proof of Lemma]
Fix generators $y_1, \dots, y_n$ for the ideal $I$  of $T$. Consider the $A$-algebra homomorphism 
$$ A[x_1, \dots, x_n] \rightarrow T \,\,\,\,\,\,\,\,\, x_i \mapsto y_i.
$$
We will show that this map induces a ring homomorphism
$$
A[[x_1, \dots, x_n]] \rightarrow \hat{T}^{I}
$$
which is finite. First note that 
for  each  natural number $t$, there is a  naturally induced  ring
homomorphism 
\begin{equation}\label{eq1} \frac{A[x_1, \dots, x_n]}{(x_1, \dots, x_n)^t} \rightarrow  \frac{T}{I^t}  
\end{equation}
sending the class $ \overline{x_i} $ to the class $\overline {y_i}.$

\smallskip
\noindent
{\bf Claim: } 
For  every  $t$, the map  (\ref{eq1}) is {\it finite}. Indeed,  if  $t_1, \dots, t_d$  are elements of $T$ whose classes  modulo $I$  are $A$-module generators for $T/I$, then the classes of $t_1, \dots, t_d$ modulo $I^t$ are generators for  $T/I^t  $  as a module over 
$ A[x_1, \dots, x_n]/ (x_1, \dots, x_n)^t.$

\medskip
The claim is straightforward to prove using induction on  $t$ and  the exact sequence
$$
0 \rightarrow I^t/I^{t+1} \rightarrow T/I^{t+1} \rightarrow T/I^t \rightarrow 0.
$$
We leave these details to the reader.

Now to prove the lemma, we take the direct limit of the maps (\ref{eq1}).  Since at every stage, the target is generated over the source by the classes of $t_1, \dots, t_d$, also in the limit, $\hat{T}^{ I}$ will be generated over $A[[x_1, \dots, x_n]]$ by the images of $t_1, \dots, t_d$. So the induced ring homomorphism $ A[[x_1, \dots, x_n]] \rightarrow T$
is finite. 
\end{proof}

\medskip
Having reduced the proof of the main results discussed in the introduction
 to Theorem \ref{main2},   the rest of the paper focuses on the local cohomology statement in the special case. 
Our proof of Theorem \ref{main2} uses an $A$-relative version of Matlis duality 
to convert the problem  to an analagous one for finitely generated modules over a power series ring, where it will follow from the theorem on generic freeness. This relative version of Matlis duality might be of interest to commutative algebraists in other contexts, and holds in greater generality than what we develop here. To keep the paper as straightforward and readable as possible, we have chosen to present it only in the case we need to prove the main result. 
Some related duality is worked out in \cite{DGI}.

\section{A Relative Matlis Dual Functor}\subsection{Matlis Duality}
We first recall the classical Matlis duality in the complete local (Noetherian) case. 

Let $(R, m)$ be a complete local ring, and let $E$ be an injective hull of its residue field $R/m$. The Matlis dual functor $\Hom_{R}(-, E)$ is an exact  contravariant  functor on the category of $R$-modules. It takes  each Noetherian  $R$-module (i.e., one satisfying the ascending chain condition) to an Artinian $R$-module (i.e., one satisfying the descending chain condition) and vice-versa.  Moreover, for any Artinian  or Noetherian 
$R$-module $\mathcal H$, we have a natural isomorphism
 $\mathcal H \rightarrow \Hom_{R}( \Hom_{R}(\mathcal H, E), E).$  
  That is, the Matlis dual functor defines an equivalence of categories between the category of Noetherian and the category of Artinian $R$-modules. 
See \cite{LocalCohomology}, \cite[Thm 3.2.13]{BrunsHerzog} or \cite{Hochster} for more on Matlis duality.

\subsection{A relative version of Matlis duality.}\label{notation}
Let $A$ be a domain.  Let $R$ be an $A$-algebra, with ideal $I\subset R$ such that $R/I$ is finitely generated as an $A$-module. 
Define the {\it relative Matlis dual functor} to be the functor
$$
\{R-modules\} \rightarrow \{R-modules\}\,\,\,\,\,\,$$
$$
M \mapsto  M^{\vee_A} :=
 \varinjlim_t \Hom_A(M/I^tM,\, A).
$$
We also denote  $M^{\vee_A}$ by $\Hom_A^{cts}(M,\, A),$ since it is the submodule of $\Hom_A(M,\, A)$  consisting of maps continuous in the  $I$-adic topology. 
That is,  $\Hom_A^{cts}(M,\, A)$ is the $R$-submodule of $\Hom_A(M,\, A)$ consisting of maps $\phi:M \rightarrow A$ satisfying $\phi(I^tM) = 0$ for some $t$.

\begin{Prop}\label{lem1}  Let $R$ be a Noetherian algebra over a Noetherian domain $A$, with ideal $I\subset R$ such that $R/I$ is finitely generated as an $A$-module.  
\begin{enumerate}
\item The   functor $Hom_A^{cts}(-, A)$ is naturally equivalent to the functor 
$$
M \mapsto  \Hom_{R}(M, \Hom_A^{cts}(R, A) ). $$ 
\item The  functor 
preserves exactness of sequences
$$ 0 \rightarrow M_1 \rightarrow M_2 \rightarrow M_3 \rightarrow 0
$$
 of finitely generated $R$-modules,
provided that the modules  $M_3/I^nM_3$ are  (locally) free $A$-modules for all $n \gg 0$.
\end{enumerate}
\end{Prop}

\begin{Rmk}
If $A = R/I$ is a  field, then the relative Matlis dual specializes to the  usual Matlis dual functor $\Hom_{R}(-, E)$, where $E$ is the injective hull of the residue field of 
$R$ at the maximal ideal $I $ (denoted here now $m$). Indeed, one easily checks that $ \Hom_A^{cts}(R, A)$ is an injective hull of $R/m$. To wit, the  $R$-module homomorphism
$$R/m \rightarrow  \Hom_A^{cts}(R, A) \,\,\,\,\,\,\,\,\,\,\,\,\,\, {\text{sending}}\,\,\,\,\,\,\,\,\,\,\,\,\,\,\,\,   r\mod  m  \mapsto 
 \begin{cases}    R \rightarrow A  \\ 
s \mapsto rs  \mod m
 \end{cases} 
$$ is a maximal essential extension of $R/m$. 
 \end{Rmk}

\begin{proof}[Proof of Proposition]
Statement (1) follows from  
  the adjointness of tensor and Hom, which is easily observed to restrict to the corresponding statement for modules of continuous maps.

  For (2), we need to
 show the sequence  remains exact after applying the relative Matlis dual functor. The functor $\Hom_A(-, A)$ preserves left exactness: that is, 
\begin{equation}\label{SEQ1}
 0 \rightarrow \Hom_A(M_3, A) \rightarrow\Hom_A(M_2, A)  \rightarrow \Hom_A(M_1, A)  
 \end{equation} 
 is exact.
 We want to show that, restricting to the submodules of  {\it continuous maps,} we also have exactness at the right. That is, we need the exactness of
\begin{equation}\label{sequence2}
 0 \rightarrow \Hom_A^{cts}(M_3, A) \rightarrow \Hom_A^{cts}(M_2, A)  \rightarrow \Hom_A^{cts}(M_1, A) \rightarrow 0.
 \end{equation}
 The exactness of  (\ref{sequence2})  at all spots except the right  is easy to verify using the description of a continuous map as one annihilated by a power of $I$.  
 
To check exactness of  (\ref{sequence2}) at the right, we use the Artin Rees Lemma  \cite[10.10]{AM}. 
 Take $\phi \in  \Hom_A^{cts}(M_1, A).$
 By definition of continuous, we 
 $\phi$ is annihilated by $I^n$ for sufficiently large $n$. 
  By  the Artin-Rees Lemma, there exists $t$ such that for all $n\geq t$, we have 
   $ I^{n+t}M_2\cap M_1 \subset I^nM_1$. This means we have a surjection
$$
M_1/ (I^{n+t}M_2\cap M_1) \twoheadrightarrow  M_1/I^nM_1.
$$
Therefore the composition 
$$
M_1/ I^{n+t}M_2\cap M_1 \twoheadrightarrow  M_1/I^nM_1 \rightarrow A
$$
gives a lifting of $\phi$ to an element $\phi'$  in $ \Hom_A(M_1/I^{n+t}M_2\cap M_1, A)$. 

Now note that for  $n \gg 0, $  we have exact sequences 
$$
0 \rightarrow M_1/M_1\cap I^{n+t}M_2 \rightarrow M_2/I^{n+t}M_2 \rightarrow M_3/I^{n+t}M_3 \rightarrow 0,$$
which are split over $A$ by our assumption that $M_3/I^{n+t}M_3 $ is projective. Thus
\begin{footnotesize}
\begin{equation}\label{sequence3}
0 \rightarrow \Hom_A( M_3/I^{n+t}M_3, A) \rightarrow  \Hom_A( M_2/I^{n+t}M_2, A) \rightarrow \Hom_A(M_1/M_1\cap I^{n+t}M_2, \, A) \rightarrow 0
\end{equation}
\end{footnotesize}
 is also split exact. This means  we can pull $\phi' \in \Hom_A( M_1/ I^{n+t}M_2\cap M_1,  \, A)$ back to some element $\tilde \phi$ in $\Hom_A(M_2/I^{n+t}M_2, A). $  So our original continuous map $M_1 \overset{\phi}\rightarrow A$ is the restriction of some map $ M_2 \overset{\tilde \phi}\rightarrow A$ which  satisfies $\tilde\phi (I^{n+t}M_2) = 0$.   This exactly says the continuous map  $\phi$ on $M_1$  extends to a continuous map  $\tilde\phi$ of $M_2$. That is, the sequence (\ref{sequence2}) is exact. 
\end{proof}

\begin{Rmk} If $M_3$ is a Noetherian module over a Noetherian algebra $R$ over the Noetherian domain $A$, then 
the assumption that $M_3/I^nM_3$ is locally free for all $n$ holds generically on $A$---that is, after inverting a single element of $A$. See Lemma \ref{moregenfree}. 
\end{Rmk}

\section{Generic Freeness}
We briefly review Grothendieck's idea of generic freeness, and use it to prove that the relative Matlis dual of a Noetherian $R$-module is generically free over $A$ (under suitable Noetherian hypothesis on $R$ and $A$). 
\smallskip

Let  $M$ be a module over a  commutative domain $A$. We say that $M$ is {\it generically free} over $A$ 
 if there exists a non-zero $g\in A$, such that  $M \otimes_A A_g$ is a free $A_g$-module, where $A_g$ denotes the localization of $A$ at the element $g$. 
 Likewise, a collection $\mathcal M$  of $A$-modules is {\it simultaneously generically free} over $A$  if there exists a non-zero $g\in A$, such that  $M \otimes_A A_g$ is a free for all 
 modules $M \in \mathcal M$. Note that any finite collection of generically free modules is always simultaneously generically free, since we can take $g$ to be the product the $g_i$ that work for each of the $M_i$.

 Of course, finitely generated $A$-modules are always generically free. 
  Grothendieck's  famous  {\bf Lemma on Generic Freeness}  ensures that many other modules  are as well:
\begin{lemma} \label{genfree} \cite[6.9.2]{EGA}
 Let $A$ be a Noetherian  domain. Let $M$ be any finitely generated module over a finitely generated $A$-algebra $T$. Then $M$ is generically free over $A$.
   \end{lemma}

We need a version of  Generic Freeness for certain infinite families of modules over more general $A$-algebras:  

 \begin{Lem}\label{moregenfree}
  Let $A$ be any domain. Let $T$ be any Noetherian $A$-algebra, and $I \subset T$ any ideal such that $T/I$ is finite over $A$. 
  Then  for any  Noetherian $T$-module $M$, 
  the  family of modules 
  $$\{ M/I^nM\, | \, n \geq 1\}$$
is simultaneously generically free  over $A$. That is,  after inverting a single element of $A$,   the modules  
  $M/I^nM$ for all  $n \geq 1$ become free  over $A$.  \end{Lem}
  
   \begin{Rmk}  We will make use of Lemma \ref{moregenfree} in the case where $T = A[[x_1, \dots, x_n]]$.
    \end{Rmk}

  \begin{proof}
  If $M$ is finitely generated over $T$, then the associated graded module
  $$ gr_IM = M/IM \oplus IM/I^2M \oplus I^2M/I^3M \oplus \dots
  $$ is finitely generated over the 
  associated graded ring $gr_IT = T/I \oplus  I/I^2 \oplus I^2/I^3 \oplus \dots$, which is a homomorphic image of a  polynomial  ring over $T/I$. Hence
   $gr_IT $ is a finitely generated $A$-algebra. 
  Applying the Lemma of generic freeness to the  $gr_IT$-module  $ gr_IM$, we see that after inverting a single  non-zero element  of  $g$ of $A$, the module 
  $ gr_IM$ becomes $A$-free. Since 
  $ gr_IM$ is graded over  $gr_IT$ and 
  $A$ acts in degree zero, clearly  its  graded pieces  are also free after tensoring with  $A_g$. 
  We can thus replace $A$ by $A_g$ for suitable $g$, and assume that the $I^nM/I^{n+1}M$ are $A_g$-free for all $n\geq 0$.

  Now consider the short exact sequences
\begin{equation}\label{finally}
 0  \rightarrow I^nM/I^{n+1}M \rightarrow M/I^{n+1}M \rightarrow M/I^{n}M \rightarrow 0,
 \end{equation}
for each $n \geq 1$.
We already know that $M/I^{1}M$ and  $I^nM/I^{n+1}M$ for all $n \geq 1$ are  free (over $A_g$), so by induction, we conclude that the sequences (\ref{finally})  are all split over $A_g$  for every $n$. In particular, the modules  $M/I^{n}M$ are also free  over $A_g$ for all $n \geq 1$.
The lemma is proved.
  \end{proof}

\begin{Prop}\label{matlisdualcommute}  Let $A$ be a Noetherian domain. Let $R$ be any Noetherian $A$-algebra with ideal $I \subset R$ such that $R/I$ is a finitely generated $A$-module. Then for any Noetherian $R$-module $M$, the relative Matlis dual 
$$
\Hom_A^{cts}(M, A)
$$
is a  generically free $A$-module. 
Furthermore, if $g \in A$ is a non-zero element such that $A_g \otimes_A 
\Hom_A^{cts}(M, A)
$ is free over $A_g$, then 
for any base change 
$A \rightarrow  L$ factoring through $A_g$,  the natural map 
$$
\Hom_A^{cts}(M, A) \otimes_A L \rightarrow 
\Hom_L^{cts}(M \otimes_A L, L)
$$
is an isomorphism of $R \otimes_A L$-modules, functorial in $M$. 
\end{Prop}

\begin{proof} We can invert one element of $A$ so that each  $M/I^tM$ is   free over $A$; replace $A$ by this localization. 
We  now claim that the $A$-module
$$
\Hom_A^{cts}(M, A) = \varinjlim_{t} 
\Hom_A\left(\frac{M}{I^tM}, A\right)$$ is  free.  Indeed, since each  $M/I^tM$ is   free  and has finite rank over $A$, its $A$-dual  $\Hom_A\left(\frac{M}{I^tM}, A\right)$ is also free of finite rank. 
The direct limit is also $A$-free because the maps in the limit system are all split over $A$ and injective. Indeed,   if some finite  $A$-linear combination of $f_i \in \Hom_A^{cts}(M, A) $ is zero, then that same combination is zero considered as elements of the free-submodule $\Hom_A\left(\frac{M}{I^tM}, A\right)$ of homomorphisms in $\Hom_A^{cts}(M, A) $ killed by  a large power of $I. $
It
 follows that
$
\Hom_A^{cts}(M, A)
$ is free over $A$, as desired.

The result on base change follows as well,  since tensor commutes with direct limits and with dualizing  a finitely generated free module. 
\end{proof}

\begin{Rmk} We can interpret Proposition \ref{matlisdualcommute} as saying that {\it generically on $A$}, the relative Matlis dual functor (applied to Noetherian $R$-modules) is exact and commutes with base change. 
\end{Rmk}

\section{Relative Local Duality and the Proof of the Main Theorem}

The proof Theorem \ref{main2} and therefore of our main result answering Koll\'ar's question, follows from a relative version of Local Duality:

\medskip
\begin{Thm}\label{locDual}  Let $R$ be a power series  ring $A[[x_1, \dots, x_n]]$ over a  Noetherian domain $A$, and let $M$ be a finitely generated $R$-module.
Then, after replacing $A$ by its localization at one element,  there is a functorial isomorphism for all $i$
$$H^i_I(M) \cong [\Ext^{n-i}_{R}(M, \, R)]^{\vee_A{}}.$$
\end{Thm}

To prove Theorem \ref{locDual}, we need the following  lemma.
\begin{lemma}
 Let $R$ be a power series  ring $A[[x_1, \dots, x_n]]$ over a ring $A$.
There is a natural $R$-module isomorphism $\Hom_A^{cts}(R, A)  \cong H^n_I(R)$, where $I = (x_1, \dots, x_n)$. In particular, the relative Matlis dual functor can also be expressed
$$
M \mapsto 
\Hom_{R}(M,  H^n_I(R)).$$ 
\end{lemma}

\begin{proof} 
We recall that  $H^i_I(R)$  is the $i$-th cohomology of 
 the extended \v{C}ech complex  $K^{\bullet}$  on the elements $x_1, \dots, x_n$. This is the  complex 
$$
0 \rightarrow R  \overset{\delta_1}\rightarrow {R}_{x_1} \oplus {R}_{x_2}  \cdots  \oplus {R}_{x_n} 
\rightarrow \bigoplus_{i<j} {R}_{x_i x_j} \rightarrow \cdots \overset{\delta_n}\longrightarrow  {R}_{x_1x_2\cdots x_n} \rightarrow 0
$$ where the maps are (sums of) suitably signed  localization maps. 
In particular, $H^n_I(R)$ is the cokernel of $\delta_n$, which can be checked to be the free $A$-module on (the classes of) the monomials $x_1^{a_1}\dots x_n^{a_n}$ with all $a_i<0$. {\footnote{See page 226 of \cite{Hartshorne}, although I have included one extra map  $\delta_1: R \rightarrow \bigoplus {R}_{x_i} $ sending  $ f \mapsto (\frac{f}{1}, \dots,  \frac{f}{1})$ in order to make the complex exact on the left,  and my  ring is a power series ring over $A$ instead of a  polynomial rings over a field. This is also discussed in \cite{LocalCohomology} page 22.}}

Now define an explicit $R$-module isomorphism  $\Phi$ from   $H^n_I(R)$  to $ \Hom_A^{cts}(R, A)$ by sending the (class of the) monomial 
$x^{\alpha}$ to the map  $\phi_{\alpha} \in  \Hom_A^{cts}(R, A) $:
$$ 
\,\,\,\,\,\,\,\,\,\,\, R  \overset{\phi_{\alpha}}\longrightarrow A \,\,\,\,\,\,\,\,\,\,\,\,\,\,\,\,\,\,\,\,\,\,\,\,\,\,\,\,\,\,\,\,\, \,\,\,\,\,\,\,\,\,\,\,\,\,\,\,\,\,\,\,\,\,\,\,\,\,\,\,\,\,\,\,\,\, \,\,\,\,\,\,\,\,\,\,\,\,\,\,\,\,\,\,\,\,\,\,\,\,\,\,\,\,\,\,\,\,\,\,\,\,\,\,\,\,\,\,\,\,\,\,\,\,\,\,\,\,\,\,\,\,\,\,\,
 $$
 $$
    x_1^{b_1} \dots x_n^{b_n} \mapsto  
    \begin{cases}    x_1^{b_1+a_1+1} \dots x_n^{b_n+a_n+1}  \mod I  
    &\mbox{if }  \alpha_i + \beta_i \geq -1  \mbox{  for all  }  i  \\ 
0  \mbox{ otherwise} \end{cases} 
 $$
 Since  $\{x^{\beta}\,|\, \beta \in \mathbb N^n\}$  is an $A$-basis for $R$,  the map $ \phi_{\alpha}$ is a well-defined $A$-module map from $R$ to $A$, and since it sends all but finitely many $x^{\beta}$ to zero, it is $I$-adically continuous. Thus the map
  $H^n_I(R) \overset{\Phi}\longrightarrow  \Hom_A^{cts}(R, A)$ is is an $A$-module 
 homomorphism; in fact, $\Phi$ is an $A$-module
 isomorphism, since it defines a bijection between the $A$-basis  $\{x^{\alpha}\,|\, \alpha_i <0\}$ for $H^n_I(R)$ and 
$\{\pi_{\beta}  \,|\, \beta_i \geq 0\} $ for  $ \Hom_A^{cts}(R, A)$ (meaning the dual basis on the free basis $x^{\beta}$ for $R$ over $A$)
 matching the indices   $\alpha_i \leftrightarrow \beta_i = -\alpha_i - 1 $.
It is easy to check that $\Phi$  is also $R$-linear, by thinking of it as ``premultiplication by $x^{\alpha+1}$" on the cokernel of  $\delta_n$. Thus $\Phi$ identifies the $R$-modules
$H^n_I(R)$ and  $ \Hom_A^{cts}(R, A).$ 
 \end{proof}

\begin{proof}[Proof of Theorem \ref{locDual}] We proceed by proving that both modules are  generically  isomorphic to a third, namely $\Tor_{n-i}^{R}(M, H^n_I(R)).$

First, recall how to compute   $H^i_I(M)$. 
Let $K^{\bullet}$ be the extended \v{C}ech complex on the elements $x_1, \dots, x_n$ 
$$
0 \rightarrow R  \overset{\delta_1}\rightarrow {R}_{x_1} \oplus {R}_{x_2}  \cdots  \oplus {R}_{x_n} 
\rightarrow \bigoplus_{i<j} {R}_{x_i x_j} \rightarrow \cdots \overset{\delta_n}\longrightarrow  {R}_{x_1x_2\cdots x_n} \rightarrow 0.
$$ 
This is a complex of flat $R$-modules, all free over $A$, exact at every spot except the right end. Thus it is a
 flat $R$-module resolution of  the local cohomology module  $H^n_I(R)$.
 The local cohomology module
$H^i_I(M)$ is the cohomology of this complex after tensoring over $R$ with $M$, that is
$$
H^i_I(M) = \Tor_{n-i}^{R}(M, H^n_I(R)).
$$

On the other hand, let us compute the relative Matlis dual of $\Ext^{n-i}_{R}(M, R)$. Let $P_{\bullet}$ be a free resolution of $M$ over $R$.  The module $\Ext^{\bullet}_{R}(M, R)$ is the cohomology of the complex 
 $\Hom_{R}(P_{\bullet}, R).$  
We would like to say that the computation of the cohomology of this complex commutes with the relative Matlis dual  functor, but the best we can say  is that this is true {\it generically on $A$}. To see this, we will apply Lemma  \ref{moregenfree} to the following finite set of $R$-modules: 
\begin{itemize}
\item  For $i = 0, \dots, n$,  the image $D_i$ of the $i$-th map of the complex $\Hom_{R}(P_{\bullet}, R)$;
\item For $i = 0, \dots, n$,  the  cohomology $Ext^{n-i}_{R}(M, R)$ of the same complex.
\end{itemize}
Lemma \ref{moregenfree} guarantees that the  modules 
$$D_i/I^t D_i \,\,\,\,\,\,{\text{and}}\,\,\,\,\,\, \Ext^{n-i}_{R}(M, R)/I^t Ext^{n-i}_{R}(M, R)$$ are all simultaneously generically free over $A$ for all $t \geq 1$. 
  This allows us to break up the complex 
 $A_g \otimes_A \Hom_{R}(P_{\bullet}, R)$ into many short exact sequences, split over $A_g$, which satisfy the hypothesis of 
Proposition \ref{lem1}(2) (using $A_g$ in place of $A$ and $A_g\otimes_A R$ in place of $R$). It follows that the computation of cohomology of  $\Hom_{R}(P_{\bullet}, R)$ commutes with the relative  Matlis dual functor (generically on $A$).
 
  Thus, after replacing $A$ by a localization at one element, 
$\Ext_{R}^{n-i}(M, R)]^{\vee_A}$ is the cohomology of the complex
$$
\Hom_{R}(\Hom_{R}(P_{\bullet}, R), H_I^n(R)).$$
In general, for any finitely generated projective module $P$ and any module  $H$ (over any Noetherian ring $R$), the natural map 
$$
P \otimes H \rightarrow Hom(Hom(P, R), H)$$
sending a simple tensor $x\otimes h$ to the map which sends $f\in Hom(P, R) $ to $f(x)h$, is an isomorphism, functorially in $P$ and $H$. 
Thus we have a natural isomorphism of complexes 
$$
 P_{\bullet} \otimes H^n_I(R) \cong \Hom_{R}(\Hom_{R}(P_{\bullet}, R), H_I^n(R)),
$$ and so
 $[\Ext^{n-i}(M,  R)]^{\vee_A}$ is identified with $\Tor_{n-i}(M, H^n_I(R))$, which as we saw is identified with $H^i_I(M)$.

Since all identifications are functorial, we have proven the relative local duality 
$H^i_I(M) \cong [\Ext^{n-i}(M, R)]^{\vee_A{}}, $ generically on $A$.

\end{proof}

\bigskip
We can finally finish the proof of Theorem \ref{free+comm.main.ptop.karen}, and hence the main result: 

\begin{proof}[Proof of Theorem \ref{main2}] Let $R$ be a power series ring over a Noetherian domain $A$, and let $M$ be any Noetherian $R$-module. We need to show that 
 the local cohomology modules 
$H^i_I(M)$ are  generically free and commute with base change over $A$.

In light of Corollary \ref{matlisdualcommute} , we can accomplish this by  showing that $H^i_I(M)$ is the relative Matlis dual of a Noetherian $R$-module, generically on $A$. 
But this is guaranteed by the relative local duality theorem Theorem \ref{locDual}, which guarantees that 
$$H^i_I(M) \cong \Ext_R^{n-i}(M, R)^{\vee A}$$
generically on $A$. 

\end{proof}
\begin{Rmk}
One could obviously develop the theory of relative Matlis duality, especially Theorem \ref{locDual},  further; I wrote down only the simplest possible case and the simplest possible statements needed to answer  Koll\'ar's question as directly as possible. 
\end{Rmk}

\end{document}